\documentclass[12pt,reqno]{amsart}

\usepackage{amssymb}

\newtheorem{theorem}{Theorem}[section]
\newtheorem{proposition}[theorem]{Proposition}
\newtheorem{lemma}[theorem]{Lemma}

\numberwithin{equation}{section}

\begin{document}

\baselineskip=15.5pt

\title[Holomorphic bundles trivializable by proper surjective holomorphic map]{Holomorphic 
bundles trivializable by proper surjective holomorphic map}

\author[I. Biswas]{Indranil Biswas}

\address{School of Mathematics, Tata Institute of Fundamental
Research, Homi Bhabha Road, Mumbai 400005, India}

\email{indranil@math.tifr.res.in}

\author[S. Dumitrescu]{Sorin Dumitrescu}

\address{Universit\'e C\^ote d'Azur, CNRS, LJAD, France}

\email{dumitres@unice.fr}

\subjclass[2010]{53C07, 32L10, 70G45}

\keywords{Finite bundle, flat holomorphic connection, finite monodromy, reductive group.}

\date{}

\begin{abstract}
Given a compact complex manifold $M$, we investigate the holomorphic vector bundles $E$ on $M$
such that $\varphi^* E$ is holomorphically trivial for some surjective holomorphic map
$\varphi$, to $M$, from some compact complex manifold. We prove that these are exactly
those holomorphic vector bundles that admit a flat holomorphic connection with finite 
monodromy homomorphism. A similar result is proved for holomorphic principal
$G$--bundles, where $G$ is a connected reductive complex affine algebraic group.

\end{abstract}

\maketitle

\section{Introduction}

Let $M$ be a compact connected complex manifold. Let $E$ be a holomorphic vector bundle on $M$ with
the following property: there is a compact connected complex manifold $X$, and a surjective holomorphic
map $\varphi\, :\, X\, \longrightarrow\, M$, such that $\varphi^*E$ is holomorphically trivial.
To clarify, the dimension of $X$ is allowed to be larger than that of $M$.
Note that if the assumption that $X$ is compact is removed, then every holomorphic vector
bundle satisfies this condition. Indeed, the pullback of $E$ to the total space of the frame
bundle for $E$ has a canonical holomorphic trivialization. However, this total space is
never compact, if $\text{rank}(E)\, >\, 0$.

Let $E$ be a holomorphic vector bundle on $M$ admitting a flat holomorphic connection $D$ whose monodromy
homomorphism
$$
\rho\, :\, \pi_1(M, \, x_0)\, \longrightarrow\, \text{GL}(E_{x_0})\, ,
$$
where $x_0\, \in\, M$ is a base point, has finite image. Consider the finite \'etale
Galois covering $f\, :\, \widetilde{M}\, \longrightarrow\, M$ corresponding to
the finite index subgroup $\text{kernel}(\rho)\, \subset\, \pi_1(M,\, x_0)$. It is easy to see
that $f^*E$ is holomorphically trivial. Therefore, $E$ satisfies the condition stated at
the beginning.

Our aim here is to prove a converse of it. More precisely, we prove the following
(see Theorem \ref{thm1}):

\begin{theorem}\label{thm0}
Let $E$ be a holomorphic vector bundle on a compact connected complex manifold $M$
satisfying the condition that there is a compact connected complex manifold $X$, and a surjective
map $\varphi\, :\, X\, \longrightarrow\, M$, such that $\varphi^*E$ is
holomorphically trivial. Then $E$ admits a flat holomorphic connection whose
monodromy homomorphism has finite image.
\end{theorem}

It is easy to see that a holomorphic vector bundle $E$ on $M$ admits a flat
holomorphic connection with finite monodromy if and only if 
there is a finite \'etale Galois covering $f\, :\, \widetilde{M}\, \longrightarrow\, M$
such that $f^*E$ is holomorphically trivial (the ``only if'' part was explained above).
In \cite{No1}, Nori characterized such vector bundles in the frame-work
of algebraic geometry. More precisely, when the base field is the field of 
complex numbers, his result gives the following statement: For an algebraic vector
bundle $V$ on a complex projective variety $Y$,
the following two conditions are equivalent:
\begin{enumerate}
\item There is a finite \'etale Galois covering $f\, :\, \widetilde{Y}\, \longrightarrow\, Y$
such that $f^*V$ is algebraically trivial.

\item There are finitely many algebraic vector bundles $W_1,\, \cdots ,\, W_\ell$ on $Y$
such that
$$
V^{\otimes k}\,=\, \bigoplus_{j=1}^\ell W^{\oplus c_{k,j}}_j
$$
for every $k\, \geq\, 1$, where $c_{k,j}$ are nonnegative integers.
\end{enumerate}

Any vector bundle satisfying the second condition is called a finite bundle
\cite[p.~35, Definition]{No1}. This definition of finiteness clearly makes sense for holomorphic
vector bundles on compact complex manifolds. A holomorphic vector bundle on a compact connected 
complex manifold is finite if and only if it admits a flat holomorphic connection with finite
monodromy \cite{Bi}.

Theorem \ref{thm0} is also extended to holomorphic principal $G$--bundles over $M$,
where $G$ is a connected reductive complex affine algebraic group (see Lemma \ref{lem1}).

An application of Lemma \ref{lem1} is given in the context of holomorphic generalized
Cartan geometries in the sense of \cite{AM,BD} (see Proposition \ref{thm2}). In the
algebraic context, a positive characteristic version of Theorem \ref{thm0} was proved
in \cite[p.~226, Theorem 1.1]{BdS}.

\section{Preliminaries}\label{se2}

Let $M$ be a connected complex manifold. The holomorphic tangent and cotangent bundles
on $M$ will be denoted by $TM$ and $\Omega^1_M$ respectively. The exterior product
$\bigwedge^i \Omega^1_M$ will be denoted by $\Omega^i_M$. 

Let $G$ be a connected complex
Lie group. The Lie algebra of $G$ will be denoted by $\mathfrak g$. Let
\begin{equation}\label{f1}
p\, :\, E\,\longrightarrow\, M
\end{equation}
be a holomorphic principal $G$--bundle on $M$. Therefore,
 $E$ is equipped with a holomorphic action of $G$ on the right which is both free
and transitive on the fibers of $p$. Consider
the holomorphic right action of $G$ on the holomorphic tangent bundle
$TE$ induced by the action of $G$ on $E$. The quotient
\begin{equation}\label{ab}
\text{At}(E)\,:=\, (TE)/G
\end{equation}
is a holomorphic vector bundle over $E/G\,=\, M$; it is called the
\textit{Atiyah bundle} for $E$. The differential $$dp\, :\,
TE\, \longrightarrow\, p^* TM$$ of the projection $p$ in \eqref{f1} is $G$--equivariant
for the trivial action of $G$ on the fibers of $p^*TM$. The action of $G$ on $E$ produces
a holomorphic homomorphism from the trivial holomorphic bundle
$$
E\times {\mathfrak g}\, \longrightarrow\, \text{kernel}(dp)
$$
which is an isomorphism. Therefore, we have a short
exact sequence of holomorphic vector bundles on $E$
\begin{equation}\label{f2}
0\,\longrightarrow\, \text{kernel}(dp)\,=\,E\times{\mathfrak g}\,\longrightarrow
\,TE \, \stackrel{dp}{\longrightarrow}\, p^*TM\, \longrightarrow\, 0
\end{equation}
in which all the homomorphisms are $G$--equivariant. The quotient $\text{kernel}(dp)/G$ is the
adjoint vector bundle $\text{ad}(E)\,=\, E({\mathfrak g})$, which is the holomorphic
vector bundle over $M$ associated to $E$ for the adjoint action of $G$ on the
Lie algebra $\mathfrak g$. Taking
quotient of the bundles in \eqref{f2}, by the actions of $G$, the following short exact sequence of
holomorphic vector bundles on $M$ is obtained:
\begin{equation}\label{f3}
0\, \longrightarrow\, \text{ad}(E)\,\longrightarrow\, \text{At}(E)\,
\stackrel{d'p}{\longrightarrow}\, TM \,\longrightarrow\, 0\, ,
\end{equation}
where $d'p$ is the descent of the homomorphism $dp$ (see \cite{At}); this exact sequence
is known as the \textit{Atiyah exact sequence} for $E$.

A holomorphic connection on $E$ is a holomorphic homomorphism of vector bundles
$$
D\, :\, TM\, \longrightarrow\, \text{At}(E)
$$
such that
$$
(d'p)\circ D\,=\, \text{Id}_{TM}\, ,
$$
where $d'p$ is the projection in \eqref{f3} (see \cite{At}). 

We note that giving a holomorphic connection on $E$ is equivalent to giving a $\mathfrak g$--valued
holomorphic $1$--form
\begin{equation}\label{om}
\omega\,\in\, H^0(E,\, \Omega^1_E\otimes_{\mathbb C}{\mathfrak g})
\end{equation}
on $E$ such that
\begin{itemize}
\item the homomorphism $\omega\, :\, TE\, \longrightarrow\, \mathfrak g$ is $G$--equivariant
for the adjoint action of $G$ on ${\mathfrak g}$, and

\item the restriction of $\omega$ to any fiber of $p$ is the Maurer--Cartan form.
\end{itemize}
The connection homomorphism $D\, :\, TM\, \longrightarrow\, \text{At}(E)$ for the
connection defined by $\omega$ is uniquely
determined by the condition that the image of $D$ coincides with the kernel of $\omega$.

The curvature of a holomorphic connection $D$ is
$$
{\mathcal K}(D)\, :=\, D\circ D\, \in\, H^0(M,\, \text{ad}(E)\otimes\Omega^2_M)\, .
$$
The connection $D$ is called \textit{flat} (or \textit{integrable}) if ${\mathcal K}(D)
\, =\, 0$. This is equivalent to the Frobenius integrability condition for the
distribution on $TE$ defined by the kernel of $\omega$ \cite[p.~78, Corollary 5.3]{KN}, \cite{Eh}.

Fix a base point $x_0\, \in\, M$ and also fix a point $z\, \in\, p^{-1}(x_0)\, \subset\, E$.
Given a flat holomorphic connection $D$ on $E$, by taking parallel translations of $z$, with respect to $D$,
along loops based at $x_0$ we obtain the monodromy homomorphism
$$
\rho(D,z)\, :\, \pi_1(M,\, x_0)\, \longrightarrow\, G\, .
$$
If we replace $z$ by $zg$, where $g\, \in\, G$, then $\rho(D,zg)(\gamma)\,=\,
g^{-1}\rho(D,z)(\gamma)g$ for all $\gamma\,\in\, \pi_1(M,\, x_0)$. If the
image of $\rho(D,z)$ is a finite group, the flat connection $D$ is said to be
having \textit{finite monodromy}.

Let $E$ be a holomorphic principal $G$--bundle over $M$ and $D$ a flat holomorphic connection on $E$
such that the corresponding monodromy homomorphism
$$
\rho(D,z)\, :\, \pi_1(M,\, x_0)\, \longrightarrow\, G
$$
has finite image. Then the subgroup $\text{kernel}(\rho(D,z))\, \subset\, \pi_1(M,\, x_0)$
defines a finite \'etale Galois connected covering
$$f \, :\, \widetilde{M}\, \longrightarrow\, M\, , 
$$
with Galois group $\text{Gal}(f)\,=\,\text{image}(\rho(D,z))$, such that
$$
f^*E\,=\, \widetilde{M}\times E_{x_0}\,=\, \widetilde{M}\times G \,;
$$
here $G$ is identified with the fiber $E_{x_0}$ of $E$ using the map $g\, \longrightarrow\, zg$,
$g\, \in\, G$. In other words, the holomorphic principal $G$--bundle $f^*E$ is
holomorphically trivial.

\section{pullback under a surjective map with connected fibers}\label{sec3.1}

Let $X$ and $M$ be compact connected complex manifolds and
\begin{equation}\label{e0}
\varphi\, :\, X\, \longrightarrow\, M
\end{equation}
a surjective holomorphic map. Let $G$ be a connected complex Lie subgroup of
$\text{GL}(N,{\mathbb C})$ for some $N\, \geq \, 1$.

\begin{proposition}\label{prop1}
Assume that for every point $x\, \in\, M$, the fiber $\varphi^{-1}(x)\, \subset\, X$
is connected. Let $E$ be a holomorphic principal $G$--bundle over $M$ such that $\varphi^* E$
is holomorphically trivial. Then $E$ is also holomorphically trivial.
\end{proposition}

\begin{proof}
Let
\begin{equation}\label{e2}
\sigma\, :\, X\, \longrightarrow\, \varphi^*E
\end{equation}
be a holomorphic section of the principal $G$--bundle $\varphi^*E$ giving a
holomorphic trivialization of it. For notational convenience, for any point $y\, \in\,
M$, the fiber $\varphi^{-1}(y)$ will be denoted by $X_y$. Consider the restriction
$(\varphi^*E)\vert_{X_y}$ of the holomorphic principal $G$--bundle $\varphi^*E$
to $X_y\, \subset\, X$. Note that $(\varphi^*E)\vert_{X_y}$ is identified
with the trivial principal $G$--bundle $X_y\times E_y$, where $E_y$ is
the fiber of $E$ over the point $y\, \in\, M$. Using this
identification between $(\varphi^*E)\vert_{X_y}$ and $X_y\times E_y$, the restriction
$\sigma\vert_{X_y}$ of the section in \eqref{e2} to $X_y$ corresponds to a
holomorphic map
\begin{equation}\label{e3}
\widehat{\sigma}_y\, :\, X_y\, \longrightarrow\, E_y\, .
\end{equation}

We note that $E_y$ is holomorphically isomorphic to $G$, and $G$ is a complex
Lie subgroup of $\text{GL}(N,{\mathbb C})$. On the other hand, $X_y$ is compact and
connected, and hence it does not admit any nonconstant holomorphic function. Therefore,
the function $\widehat{\sigma}_y$ in \eqref{e3} is a constant map. Consequently, the map $\sigma$
in \eqref{e2} descends to a holomorphic section of $E$. In other words, there is
a holomorphic section
$$
\sigma'\, :\, M\, \longrightarrow\, E
$$
such that $\varphi^*\sigma'\,=\, \sigma$. This section $\sigma'$ produces a holomorphic
trivialization of $E$.
\end{proof}

It should be mentioned that Proposition \ref{prop1} is not valid if the assumption --- that
$G$ is a complex Lie subgroup of $\text{GL}(N,{\mathbb C})$ for some $N\, \geq \, 1$ --- is
removed. To see this, let $\mathbb T$ be a compact complex torus, and let
$$
\phi\, :\, F_{\mathbb T}\, \longrightarrow\, M
$$
be a nontrivial holomorphic principal $\mathbb T$--bundle; see \cite{Ho} for nontrivial holomorphic
torus bundles and their properties. Now, the fibers of $\phi$ are connected, and the
principal $\mathbb T$--bundle $\phi^*F_{\mathbb T}$ has a tautological holomorphic trivialization.

\section{Pullback of holomorphic principal bundles}

\subsection{Pullback of holomorphic vector bundles}\label{se4.1}

An open dense subset $\mathcal W$ of a connected complex manifold $\mathcal Z$ will be called a
nonempty Zariski open subset if there are finitely many closed complex analytic subsets ${\mathcal
S}_k\, \subset\,\mathcal Z$, $1\,\leq\, k\, \leq\, N$, with $\dim {\mathcal S}_k\, <\,
\dim \mathcal Z$ for all $k$, such that
$$
{\mathcal Z}\setminus \mathcal W\,=\, \bigcup_{k=1}^N {\mathcal S}_k\, .
$$

As before $X$ and $M$ are compact connected complex manifolds, and $\varphi$, as in \eqref{e0}, is a 
surjective holomorphic map. We no longer assume that the fibers of $\varphi$ are connected.

\begin{theorem}\label{thm1}
Let $V$ be a holomorphic vector bundle on $M$ such that the holomorphic vector
bundle $\varphi^*V$ is trivial. Then $V$ admits a flat holomorphic connection
of finite monodromy.
\end{theorem}

\begin{proof}
Since $\varphi$ is surjective, we have a natural inclusion of coherent analytic sheaves
\begin{equation}\label{iota}
\iota\, :\, {\mathcal O}_M\, \hookrightarrow\, \varphi_*{\mathcal O}_X\, .
\end{equation}
We will show that $\iota({\mathcal O}_M)$ is a direct summand of $\varphi_*{\mathcal O}_X$, meaning
there is a coherent analytic sheaf $\mathbb S$ on $M$ such that
\begin{equation}\label{e4}
\varphi_*{\mathcal O}_X\, =\, \iota({\mathcal O}_M)\oplus {\mathbb S}\, .
\end{equation}
To prove \eqref{e4}, let
$$
X\, \stackrel{\beta}{\longrightarrow}\, Z \, \stackrel{\gamma}{\longrightarrow}\, M
$$
be the Stein factorization of the map $\varphi$ (see \cite[p.~213]{GR} for Stein factorization).
We recall that this means that $\gamma$ is a finite map and $\beta_*{\mathcal O}_X\,=\, {\mathcal O}_Z$
\cite[p.~213]{GR}. We note that $Z$ is a normal space because $X$ is normal (see (2) of Stein factorization
theorem in \cite[p.~213]{GR}). Since all the fibers of $\beta$ are connected (see (2) of Stein
factorization theorem in \cite[p.~213]{GR}), we have
\begin{equation}\label{ei}
\gamma_*{\mathcal O}_Z\,=\, (\gamma\circ\beta)_*{\mathcal O}_X\,=\, \varphi_*{\mathcal O}_X\, .
\end{equation}
There is a nonempty Zariski open subset ${\mathcal U}\, \subset\, M$ such that
\begin{itemize}
\item $\gamma^{-1}({\mathcal U})\, =:\,
\widetilde{\mathcal U} $ is smooth, and

\item the complex codimension of the complement $M\setminus {\mathcal U}\, \subset\,
M$ is at least two.
\end{itemize}
In fact, ${\mathcal U}$ can be taken to be the open subset such that $M\setminus {\mathcal U}$
is the image, under the map $\gamma$, of the singular locus of $Z$.

Over ${\mathcal U}$, we have the trace map
\begin{equation}\label{tp}
\tau'\, :\, \gamma_*{\mathcal O}_{\widetilde{\mathcal U}}\, \longrightarrow\, {\mathcal O}_{\mathcal U}\, .
\end{equation}
We recall that $\tau'$ is defined as follows:
for any open subset ${\mathcal W}\, \subset\, {\mathcal U}$, and any
function $f\, \in\, H^0(\gamma^{-1}({\mathcal W}), \, {\mathcal O}_{\gamma^{-1}({\mathcal W})})$,
$$
\tau'(f)(m)\, =\, \sum_{z\in \gamma^{-1}(m)} f(z)
$$
for all $z\,\in\, {\mathcal W}$, where the summation is over the points of the inverse image
$\gamma^{-1}(m)$ with multiplicity. It is straightforward to check that the composition of homomorphisms
\begin{equation}\label{ec}
{\mathcal O}_{\mathcal U} \, \longrightarrow\, \gamma_*{\mathcal O}_{\widetilde{\mathcal U}}
\, \stackrel{\tau'}{\longrightarrow}\, {\mathcal O}_{\mathcal U}\, ,
\end{equation}
where ${\mathcal O}_{\mathcal U} \, \longrightarrow\, \gamma_*{\mathcal O}_{\widetilde{\mathcal U}}$
is the natural homomorphism as in \eqref{iota}, coincides with multiplication by the
degree of the map $\gamma$.

Now, since ${\mathcal O}_M$ is locally free, and the complex codimension of the complement
$M\setminus {\mathcal U}$ is at least two, using Hartogs' extension theorem, the
homomorphism $\tau'$ in \eqref{tp} extends uniquely to a homomorphism
\begin{equation}\label{tp2}
\tau\, :\, \gamma_*{\mathcal O}_Z\, \longrightarrow\, {\mathcal O}_M\, .
\end{equation}
As the composition of homomorphisms in \eqref{ec} is
multiplication by the degree of $\gamma$, and any endomorphism of ${\mathcal O}_M$ is multiplication
by a constant function, it follows immediately that the
composition of homomorphisms
\begin{equation}\label{ec2}
{\mathcal O}_M \, \longrightarrow\, \gamma_*{\mathcal O}_Z
\, \stackrel{\tau}{\longrightarrow}\, {\mathcal O}_M\, ,
\end{equation}
where ${\mathcal O}_M \, \longrightarrow\, \gamma_*{\mathcal O}_Z$
is the natural homomorphism as in \eqref{iota}, coincides with multiplication by the
degree of the map $\gamma$.

Let
$$
{\mathbb S}\, \subset\, \varphi_*{\mathcal O}_X
$$
be the subsheaf that corresponds to $\text{kernel}(\tau)\, \subset\, \gamma_*{\mathcal O}_Z$ (see \eqref{tp2})
by the isomorphism in \eqref{ei}. From the above observation, that the composition of homomorphisms
in \eqref{ec2} coincides with the multiplication by ${\rm degree}(\gamma)$, it follows immediately that
the isomorphism in \eqref{e4} holds.

Tensoring both sides of \eqref{e4} by $V$ we get that
\begin{equation}\label{e5}
V\otimes \varphi_*{\mathcal O}_X\, =\, V\oplus (V\otimes {\mathbb S})\, .
\end{equation}
On the other hand, by the projection formula,
\begin{equation}\label{e6}
V\otimes \varphi_*{\mathcal O}_X\,=\, \varphi_*\varphi^*V\, .
\end{equation}
Since $\varphi^*V\, =\, {\mathcal O}^{\oplus r}_X$, where $r$ is the
rank of $V$, combining \eqref{e5} and \eqref{e6} we conclude that
\begin{equation}\label{vd}
V\oplus (V\otimes {\mathbb S})\, =\, \varphi_*{\mathcal O}^{\oplus r}_X\,=\,
(\varphi_*{\mathcal O}_X)^{\oplus r}\, .
\end{equation}
In particular, the holomorphic vector bundle $V$ is a direct summand of
$(\varphi_*{\mathcal O}_X)^{\oplus r}$.

We now recall a result of Atiyah in \cite{At0}.
Any torsionfree coherent analytic sheaf ${\mathcal E}$ on $M$ can be expressed as
$$
{\mathcal E}\,=\, \bigoplus_{i=1}^\ell {\mathcal E}_i\, ,
$$
where each ${\mathcal E}_i$, $1\, \leq\, i\, \leq\, \ell$, is an indecomposable
torsionfree coherent analytic sheaf on $M$,
and ${\mathcal E}_i$, $1\, \leq\, i\, \leq\, \ell$, are unique
up to a permutation of $\{1,\, \cdots, \, \ell\}$ \cite[p.~315, Theorem 2]{At0}. We shall apply this
theorem of Atiyah to $\varphi_*{\mathcal O}_X$, and we shall separate, for our convenience,
the direct summands which are locally free and those which are not locally free.

So $\varphi_*{\mathcal O}_X$ is expressed as
\begin{equation}\label{e7}
\varphi_*{\mathcal O}_X\,=\, \left(\bigoplus_{i=1}^m \mathcal{W}_i\right)\bigoplus
\left(\bigoplus_{j=1}^n {\mathcal F}_j\right)\, ,
\end{equation}
where each $\mathcal{W}_i$, $1\, \leq\, i\, \leq\, m$, is an indecomposable holomorphic vector 
bundle on $M$ and each ${\mathcal F}_j$, $1\, \leq\, j\, \leq\, n$, is an indecomposable torsionfree 
coherent analytic sheaf on $M$ which is not locally free. In this case, the above theorem of Atiyah says that
$\{\mathcal{W}_1,\, \cdots ,\, 
\mathcal{W}_m\}$ are unique up to a permutation of $\{1,\, \cdots, \, m\}$, and $\{{\mathcal 
F}_1,\, \cdots ,\, {\mathcal F}_n\}$ are unique up to a permutation of $\{1,\, \cdots, \, n\}$.

We noted above that the holomorphic vector bundle $V$ is a direct summand of $(\varphi_*{\mathcal
O}_X)^{\oplus r}$. In view of this and \eqref{e7}, from the above theorem of Atiyah it can be
deduced that $V$ is a direct sum of copies of $\{\mathcal{W}_1,\, \cdots ,\, \mathcal{W}_m\}$, meaning
\begin{equation}\label{em}
V\,=\, \bigoplus_{i=1}^m \mathcal{W}^{\oplus d_i}_i\, ,
\end{equation}
where $d_i$ are nonnegative integers; by convention, $\mathcal{W}^{\oplus 0}_i\,=\,0$. To see
\eqref{em}, first note that from \eqref{vd} and \eqref{e7},
$$
V\oplus (V\otimes {\mathbb S})\, =\, \left(\bigoplus_{i=1}^m \mathcal{W}_i\right)^{\oplus r}\bigoplus
\left(\bigoplus_{j=1}^n {\mathcal F}_j\right)^{\oplus r}\, .
$$
Now expressing $V$ as a direct sum of indecomposable holomorphic vector bundles, and
$V\otimes {\mathbb S}$ as a direct sum of torsionfree indecomposable
coherent analytic sheaves, we conclude from the uniqueness part
of the above theorem of Atiyah that $V$ is a direct summand of
$\left(\bigoplus_{i=1}^m \mathcal{W}_i\right)^{\oplus r}\,=\, \bigoplus_{i=1}^m \mathcal{W}^{\oplus r}_i$.

Since $\varphi^*V$ is a holomorphically trivial vector bundle, it follows that 
$\varphi^*V^{\otimes k}\,=\, (\varphi^*V)^{\otimes k}$ is also a holomorphically trivial 
vector bundle for every integer $k\, \geq\, 1$. Consequently, substituting $V^{\otimes k}$ for 
$V$ in the above argument we conclude that $V^{\otimes k}$ is a direct sum of copies of 
$\{\mathcal{W}_1,\, \cdots ,\, \mathcal{W}_m\}$, for every $k\, \geq\, 1$.

We recall from \cite{No1} and \cite{No2} that a holomorphic vector bundle $\mathcal E$ on $M$ is called
a \textit{finite bundle} if there are finitely many holomorphic vector bundles
$B_1,\, \cdots,\, B_\ell$ on $M$ such that for every integer $k\, \geq\, 1$,
$$
{\mathcal E}^{\otimes k}\,=\, \bigoplus_{j=1}^\ell B^{\oplus c_{k,j}}_j\, ,
$$
where $c_{k,j}$ are nonnegative integers (\cite[p.~35, Definition]{No1}, \cite[p.~35, Lemma 3.1(d)]{No1}), 
\cite[p.~80, Definition]{No2}, \cite[(2.1)]{Bi}.

The vector bundle $V$ is finite, because $V^{\otimes k}$ is a direct sum of copies of
$\{\mathcal{W}_1,\, \cdots ,\, \mathcal{W}_m\}$, for every $k\, \geq\, 1$.
Now Theorem 1.1 of \cite{Bi} says that $V$ 
admits a flat holomorphic connection with finite monodromy.
\end{proof}

\subsection{Reductive structure group}

Let $G$ be a connected reductive complex affine algebraic group. As in \eqref{f1},
$$
p\, :\, E\,\longrightarrow\, M
$$
is a holomorphic principal $G$--bundle over a compact complex manifold $M$.
Take $(X,\, \varphi)$ as in Section \ref{se4.1}.

\begin{lemma}\label{lem1}
If the holomorphic principal $G$--bundle $\varphi^*E$ is holomorphically trivial,
then $E$ admits a flat holomorphic connection with finite monodromy.
\end{lemma}

\begin{proof}
Let
$$
\rho\, :\, G\, \longrightarrow\, \text{GL}(\mathbb{V})
$$
be a faithful algebraic representation, where $\mathbb V$ is a finite dimensional
complex vector space. Let $E_{\mathbb V}\, :=\, E({\mathbb V})$ be the holomorphic
principal $\text{GL}(\mathbb{V})$--bundle obtained by extending the structure group of
$E$ using the homomorphism $\rho$.

Since the group $G$ is reductive, the homomorphism of $G$--modules
$$
\rho\, :\, {\mathfrak g}\, \longrightarrow\, 
\text{End}({\mathbb V})\,=\, \text{Lie}(\text{GL}(\mathbb{V}))
$$
splits \cite[p.~128, Theorem 9.19]{FH}. Fix such a splitting; let
\begin{equation}\label{th}
\theta\, :\, \text{End}({\mathbb V})\,\longrightarrow\,\mathfrak g
\end{equation}
be the projection of $G$--modules corresponding to the chosen splitting.

Assume that $\varphi^*E$ is holomorphically trivial. Then the holomorphic
principal $\text{GL}(\mathbb{V})$--bundle $\varphi^*E_{\mathbb V}$ is also
holomorphically trivial. Now Theorem
\ref{thm1} says that $E_{\mathbb V}$ admits a flat holomorphic connection $D$ with finite
monodromy. Let
\begin{equation}\label{om2}
\omega\,\in\, H^0(E_{\mathbb V},\, \Omega^1_{E_{\mathbb V}}\otimes_{\mathbb C}
\text{End}({\mathbb V}))
\end{equation}
be a section as in \eqref{om} giving a flat holomorphic connection on $E_{\mathbb V}$ with
finite monodromy.

Let
$$
\Phi\, :\, E\, \hookrightarrow\, E_{\mathbb V}
$$
be the natural inclusion map. Consider the $\mathfrak g$--valued holomorphic
$1$--form
$$
\theta\circ \Phi^*\omega\, \in\, H^0(E,\, \Omega^1_E\otimes_{\mathbb C}{\mathfrak g})\, ,
$$
where $\theta$ is the homomorphism in \eqref{th} and $\omega$ is the $1$--form in \eqref{om2}.
It is straightforward to check that $\theta\circ \Phi^*\omega$ is a holomorphic connection on $E$.

If ${\mathcal K}(\omega)\, \in\, H^0(M,\, \text{ad}(E_{\mathbb V})\otimes\Omega^2_M)$
is the curvature of the connection given by $\omega$, then the curvature of the
connection on $E$ given by $\theta\circ\Phi^*\omega$ is
$$
\theta\circ {\mathcal K}(\omega)\, \in\, H^0(M,\, \text{ad}(E)\otimes\Omega^2_M)\, .
$$
Therefore, the connection given by $\theta\circ\Phi^*\omega$ is flat, because
${\mathcal K}(\omega)\,=\, 0$. The
monodromy of the flat connection defined by $\theta\circ\Phi^*\omega$ is finite, because
the connection defined by $\omega$ has finite monodromy.
\end{proof}

\section{Holomorphic generalized Cartan geometry}

Let $G$ be a connected complex Lie group and $H\, \subset\, G$ a closed connected complex
Lie subgroup. Denote by $\mathfrak g$ the Lie algebra of $G$.

There is a standard notion of Cartan geometry \cite{Sh}. Roughly speaking a Cartan geometry 
with model $(G,\,H)$ is infinitesimally modeled on the homogeneous space $G/H$. The Cartan 
geometry is flat (i.e., has vanishing curvature) if and only if it is locally isomorphic (not 
just infinitesimally) to the homogeneous space $G/H$ in the sense of Ehresmann \cite{Eh} (see 
also \cite{Go}).

It is a very stringent condition for a compact complex manifold to admit a holomorphic Cartan 
geometry. A more flexible notion of {\it generalized holomorphic Cartan geometry} was 
introduced in \cite{BD} (see also \cite{AM}). Holomorphic generalized Cartan 
geometries are stable under pullback by holomorphic maps \cite{BD}.

A generalized holomorphic Cartan geometry of model $(G,\,H)$ on a complex manifold $M$ is given 
by a holomorphic principal $H$--bundle $E_H$ over $M$ endowed with a $\mathfrak{g}$-valued 
holomorphic one form on $E_H$ such that the following two hold:
\begin{enumerate}
\item $\omega$ is $H$--equivariant with $H$ acting on $\mathfrak g$ via adjoint representation;

\item the restriction of $\omega$ to each fiber of $E_H$ coincides with the Maurer--Cartan form
associated to the action of $H$ on $E_H$.
\end{enumerate}
A generalized Cartan geometry is called flat if its curvature ${\mathcal K}(\omega)\,=\, d \omega + 
\frac{1}{2} \lbrack \omega, \omega \rbrack_{\mathfrak g}$ vanishes identically.

Notice that the standard definition of a Cartan geometry requires that $\omega$ realizes a 
pointwise linear isomorphism between $TE_H$ and $\mathfrak g$ (which implies that the complex 
dimension of $M$ coincides with that of $G/H$).

Denote by $E_G$ the holomorphic principal $G$--bundle constructed from $E_H$ by extension of the 
structure group using the inclusion map $H \,\hookrightarrow\, G$. Denote by $ {\rm ad}(E_H)$ and 
${\rm ad}(E_G)$ the adjoint bundles of $E_H$ and $E_G$ respectively. Let ${\rm At}(E_H)$ be the 
Atiyah bundle for $E_H$ (see \eqref{ab}).

A generalized holomorphic Cartan geometry of model $(G,\,H)$ is equivalently defined (see 
\cite{BD}) by a homomorphism $\Psi \,:\, {\rm At}(E_H) \,\longrightarrow\, \mathfrak{g}$
such that the following diagram commutes:
\begin{equation}\label{e6a}
\begin{matrix}
0 & \longrightarrow & {\rm ad}(E_H) & \longrightarrow & {\rm At}(E_H)
& \longrightarrow & TM & \longrightarrow & 0\\
&& \Vert && ~\Big\downarrow \Psi && ~\Big\downarrow t \\
0 & \longrightarrow & {\rm ad}(E_H) & {\longrightarrow} & {\rm ad}(E_G)
& \longrightarrow & {\rm ad}(E_G)/{\rm ad}(E_H) & \longrightarrow & 0
\end{matrix}
\end{equation}
Notice that the top row of the diagram is the exact sequence in \eqref{f3} corresponding to the 
principal bundle $E_H$ and the bottom row is given by the canonical inclusion of $ {\rm 
ad}(E_H)$ in ${\rm ad}(E_G)$. The homomorphism $t$ in \eqref{e6a} is is uniquely defined by 
$\Psi$ and the commutativity of the diagram.

The homomorphism $\Psi$ in \eqref{e6a} defines a canonical connection $D_G$ on the principal $G$-bundle $E_G$ and the 
generalized Cartan geometry is flat (i.e., $\mathcal {K}(\omega)\,=\,0$) if and only if the connection $D_G$ is flat (see 
Proposition 3.4 and Section 3.3 in \cite{BD}). In this case the monodromy of $(E_G,\, D_G)$ is called the monodromy of 
the generalized Cartan geometry.

Now let $G$ be a connected reductive complex affine algebraic group and $H\, \subset\,
G$ a connected closed algebraic subgroup of it. In the context of generalized
Cartan geometries, Lemma \ref{lem1} has the following consequence.

\begin{proposition}\label{thm2}
Take $M$, $X$ and $\varphi$ as \eqref{e0}. Let $(E_H,\, \omega)$ be a holomorphic generalized 
Cartan geometry on $M$, with model $(G,\,H)$, such that the pulled back generalized Cartan 
geometry $(\varphi^*E_H, \,\varphi^* \omega)$ on $X$ is flat and has trivial monodromy 
homomorphism. Then $M$ admits a flat holomorphic generalized Cartan geometry $(E_H, \,\omega')$, with 
model $(G,H)$, whose monodromy homomorphism has finite image.
\end{proposition}

\begin{proof}
We apply Lemma \ref{lem1} to the principal $G$--bundle $E_G$. Since the 
generalized Cartan geometry $(\varphi^*E_H, \,\varphi^* \omega)$ is flat with trivial monodromy, 
it follows that $\varphi^*E_G$ is flat with trivial monodromy (note that $\varphi^*E_G$ coincides 
with the principal $G$-bundle associated to $\varphi^*E_H$ by extension of the structure 
group). Consequently, $\varphi^*E_G$ is holomorphically trivial. Now Lemma \ref{lem1} implies 
that $E_G$ admits a holomorphic flat connection $D'_G$ with finite monodromy. Together with the 
reduction of the structure group $E_H\, \subset\, E_G$ to the subgroup $H$ the connection $D'_G$ defines a 
holomorphic generalized geometry $(E_H,\, \omega')$ on $M$ \cite{BD} (Theorem 3.7, point (3)). 
Since $D'_G$ is flat, the curvature ${\mathcal K}(\omega')$ of $\omega'$ vanishes identically. The 
monodromy of $(E_H, \,\omega')$ coincides with that of $(E_G, D'_G)$, so it is finite. 
\end{proof}

\section*{Acknowledgements}

This work has been supported by the French government through the UCAJEDI Investments in the 
Future project managed by the National Research Agency (ANR) with the reference number 
ANR2152IDEX201. The first-named author is partially supported by a J. C. Bose Fellowship, and 
school of mathematics, TIFR, is supported by 12-R$\&$D-TFR-5.01-0500. The second-named author 
wishes to thank TIFR Mumbai for hospitality, while the first-named author thanks
Universit\'e C\^ote d'Azur for hospitality.


\end{document}